\documentclass[12pt]{amsart}

\usepackage{amsthm}
\usepackage{courier}

\usepackage[cmtip,arrow]{xy}
\usepackage{pb-diagram, pb-xy}

\newtheorem{theorem}{Theorem}
\newtheorem{corollary}[theorem]{Corollary}
\newtheorem{lemma}[theorem]{Lemma}

\newtheorem{remark}[theorem]{Remark}

\newcommand{\A}{\mathcal{A}}

\begin{document}
\title{Acyclic complexes and Gorenstein rings}
\author{Sergio Estrada}
\address[S. Estrada]{Departamento de Matem\'aticas. Universidad de Murcia. Murcia 30100. ESPA\~{N}A}
\email{sestrada@um.es}

\author{Alina Iacob}
\address[A. Iacob]{Department of Mathematical Sciences. Georgia Southern University. Statesboro (GA) 30460-8093. USA}
\email{aiacob@GeorgiaSouthern.edu}

\author{Holly Zolt}
\address[H. Zolt]{Department of Mathematical Sciences. Georgia Southern University. Statesboro (GA) 30460-8093. USA}
\email{hz00179@georgiasouthern.edu}


\begin{abstract}
For a given class of modules $\A$, we denote by $\widetilde{\A}$ the class of exact complexes $X$ having all cycles in $\A$, and by $dw(\A)$ the class of complexes $Y$ with all components $Y_j$ in $\A$. We use the notations $\mathcal{GI}$ $(\mathcal{GF}, \mathcal{GP})$ for the class of Gorenstein injective (Gorenstein flat, Gorenstein projective respectively) $R$-modules, $\mathcal{DI}$ for Ding injective modules, and $\mathcal{PGF}$ for projectively coresolved Gorenstein flat modules (see section 2 for definitions). We prove that the following are equivalent over any ring $R$:
\begin{enumerate}
\item Every exact complex of injective modules is totally acyclic.
\item Every exact complex of Gorenstein injective modules is in $\widetilde{\mathcal{GI}}$.
\item Every complex in $dw(\mathcal{GI})$ is dg-Gorenstein injective.
\end{enumerate}
Then we show that the analogue result for complexes of flat and Gorenstein flat modules also holds over arbitrary rings.
We also prove that if moreover, the ring is $n$-perfect for some integer $n \ge 0$, then the three equivalent statements for flat and Gorenstein flat modules are also equivalent with their counterparts for projective and projectively coresolved Gorenstein flat modules.
We also prove the following characterization of Gorenstein rings:
Let $R$ be a commutative coherent ring. The following statements are equivalent:
(1) every exact complex of FP-injective modules has all its cycles Ding injective modules.
(2) every exact complex of flat modules is F-totally acyclic, and every $R$-module M such that $M^+$ is Gorenstein flat is Ding injective.
(3) every exact complex of injectives has all its cycles Ding injective modules
and every $R$-module M such that $M^+$ is Gorenstein flat is Ding injective.
If moreover the ring $R$ has finite Krull dimension then statements (1), (2)
and (3) above are also equivalent to
(4) $R$ is a Gorenstein ring (in the sense of Iwanaga).
\end{abstract}
\maketitle

\section{introduction}

Gorenstein homological algebra is the relative version of homological algebra that uses Gorenstein projective, Gorenstein injective and Gorenstein flat resolutions instead of the classical projective, injective and flat resolutions. The Gorenstein methods have proved to be very useful in commutative and non commutative algebra, as well as in representation theory and model category theory. This is the reason why the existence of the Gorenstein resolutions, and the properties of the Gorenstein modules have been studied intensively in recent years. The Gorenstein injective, projective, respectively flat modules are defined in terms of totally acyclic (respectively F-totally acyclic) complexes. We recall that an exact complex of injective (projective respectively) modules is called totally acyclic if it stays exact when applying a functor $\mathrm{Hom}(A,-)$, with $A$ injective, to it (respectively when applying to the complex a functor $\mathrm{Hom}(-,P)$ with $P$ any projective module). And, an exact complex of flat modules is called F-totally acyclic if it stays exact when tensoring it with any injective module. A module is said to be Gorenstein injective (Gorenstein projective) if it is a cycle in a totally acyclic complex of injective (projective) modules. The Gorenstein flat modules are the cycles of F-totally acyclic complexes.

By a Gorenstein ring we mean Gorenstein in the sense of Iwanaga (i.e. a two sided noetherian ring of finite self injective dimension on both sides). It is known that over a Gorenstein ring every exact complex of injective (projective, respectively) modules is totally acyclic, and every exact complex of flat modules is F-totally acyclic (see Section 2 for details). So, over such a ring, the Gorenstein injective (projective, flat respectively) modules are simply the cycles of the exact complexes of injective (projective, flat respectively) modules. Thus, it is a natural question to consider whether or not these conditions characterize Gorenstein rings, and, more generally, if it is possible to characterize Gorenstein rings in terms of exact complexes of Gorenstein injective, Gorenstein projective, Gorenstein flat modules. We considered first these questions in \cite{iacob:17:totally}, over noetherian rings only. We extend here some of the results of \cite{iacob:17:totally} to arbitrary rings. We prove, for instance (Theorem 4), that over any ring $R$, the following statements are equivalent:
\begin{enumerate}
\item Every exact complex of injective modules is totally acyclic.
\item Every exact complex of Gorenstein injective modules is in $\widetilde{\mathcal{GI}}$ (see section 2 for definitions and notations).
\item Every complex of Gorenstein injective modules is a dg-Gorenstein injective complex.
\end{enumerate}
Theorem 7 shows that the dual results for complexes of flat and Gorenstein flat modules also holds over any ring. If, moreover, the ring is left $n$-perfect for some nonnegative integer $n$, then the three statements from Theorem 7 are also equivalent with their counterparts for complexes of projective and $\mathcal{PGF}$ modules (this is Theorem 12). The projectively coresolved Gorenstein flat modules ($\mathcal{PGF}$ modules, for short) mentioned here were introduced by Saroch and Stovicek in \cite{saroch:18:gor}; they are the cycles of exact complexes of projective modules that stay exact when tensored with any injective module (see section 2 for details).\\

It is known (\cite{iacob:17:totally}, Corollary 2) that if the ring $R$ is commutative noetherian of finite Krull dimension then the three equivalent conditions in Theorem 4 are equivalent to those in Theorem 7 as well as with their projective/Gorenstein projective counterparts, and they characterize Gorenstein rings. Consequently, when $R$ is commutative noetherian of finite Krull dimension, the statements from Theorem 4 and those from Theorem 12 are all equivalent and they hold if and only if the ring is Gorenstein.

We also consider here a generalization of the class of Gorenstein injective modules - the Ding injective modules. They are the cycles of exact complexes of injective modules that remain exact when applying a functor $\mathrm{Hom}(A,-)$ for any FP-injective module $A$. We prove (Theorem 16) that over a coherent ring $R$ the following statements are equivalent ($n < \infty$ is a non negative integer):
\begin{enumerate}
\item Every exact complex of FP-injective modules has all its cycles of Ding injective dimension at most $n$.
\item Every exact complex of injective modules has all its cycles Ding injective modules, and, if $M^+$ is Gorenstein flat, then the module $M$ has Ding injective dimension at most $n$.
\item Every exact complex of flat modules is F-totally acyclic, and, if $M^+$ is a Gorenstein flat module, then the module $M$ has Ding injective dimension at most $n$.
\end{enumerate}

We end with the following result, which gives, in particular, a characterization of Gorenstein rings (Theorem 17):\\
Let $R$ be a commutative coherent ring. The following statements are equivalent:
\begin{enumerate}
\item Every exact complex of FP-injective modules has all its cycles Ding injective modules.
\item Every exact complex of injectives has all its cycles Ding injective modules and if $M^+$ is a Gorenstein flat module, then $M$ is Ding injective.
\item Every exact complex of flat modules is F-totally acyclic, and if $M^+$ is a Gorenstein flat module then $M$ is a Ding injective module.
\end{enumerate}
If moreover the ring $R$ has finite Krull dimension then statements (1), (2) and (3) above are also equivalent to:\\
(4) $R$ is a Gorenstein ring.

\section{preliminaries}

Throughout the paper $R$ will denote an associative ring with unit (non necessarily commutative). The category of left $R$-modules will be denoted by $R\textrm{-Mod}$, and the category of unbounded complexes of left $R$-modules will be denoted by $Ch(R)$. Modules are, unless otherwise explicitly stated, left modules.

We recall that an acyclic complex of projective modules $$P = \ldots \rightarrow P_1 \rightarrow P_0 \rightarrow P_{-1} \rightarrow \ldots $$ is said to be \emph{totally acyclic } if for any projective $R$-module $Q$, the complex $$\ldots \rightarrow \mathrm{Hom}(P_{-1}, Q) \rightarrow \mathrm{Hom}(P_0, Q) \rightarrow \mathrm{Hom}(P_1, Q) \rightarrow \ldots$$ is still acyclic.

A module $G$ is Gorenstein projective if it is a cycle of such a totally acyclic complex of projective modules ($G = Z_j(P)$, for some integer $j$).\\

Dually, an acyclic complex of injective modules $$I= \ldots \rightarrow I_1 \rightarrow I_0 \rightarrow I_{-1} \rightarrow \ldots $$ is said to be \emph{totally acyclic } if for any injective $R$-module $E$, the complex $$\ldots \rightarrow \mathrm{Hom}(E, I_1) \rightarrow \mathrm{Hom}(E, I_0) \rightarrow \mathrm{Hom}(E, I_{-1}) \rightarrow \ldots$$ is still acyclic.

A module $M$ is Gorenstein injective if it is a cycle in a totally acyclic complex of injective modules ($M =  Z_j(I)$ for some integer $j$).\\

We also recall that an acyclic complex of flat left $R$-modules, $$F = \ldots \rightarrow F_1 \rightarrow F_0 \rightarrow F_{-1} \rightarrow \ldots $$ is said to be \emph{F-totally acyclic} if for any injective right $R$-module $I$, the complex $$\ldots \rightarrow I \otimes F_1 \rightarrow I \otimes I_0 \rightarrow I \otimes F_{-1} \rightarrow \ldots$$ is still acyclic.\\
A module $H$ is Gorenstein flat if it is a cycle in an F-totally acyclic complex of flat modules ($H = Z_j(F)$ for some integer $j$).\\

\begin{remark} If $R$ is an $n$-Gorenstein ring then, by \cite{enochs:00:relative}, Theorem 12.3.1, every $n$th injective cosyzygy of an $R$-module is Gorenstein injective. So, for any acyclic complex of injective modules, $I$, using the result for its $(j-n)$th cycle, we obtain that $Z_j I$ is Gorenstein injective, for all $j$. Thus, over a Gorenstein ring $R$, every acyclic complex of injective modules is totally acyclic.\\
The same type of argument and another use of \cite{enochs:00:relative}, Theorem 12.3.1, for projective (flat respectively) syzygies, gives that, over a Gorenstein ring, every acyclic complex of projective (flat respectively) modules is totally acyclic (F-totally acyclic respectively).
\end{remark}

The Ding injective modules are also defined in terms of exact complexes of injectives. More precisely, a module $D$ is called \emph{Ding injective} if it is a cycle of an acyclic complex of injective modules $$I= \ldots \rightarrow I_1 \rightarrow I_0 \rightarrow I_{-1} \rightarrow \ldots $$ such that for any FP-injective $R$-module $A$, the complex $$\ldots \rightarrow \mathrm{Hom}(A, I_1) \rightarrow \mathrm{Hom}(A, I_0) \rightarrow \mathrm{Hom}(A, I_{-1}) \rightarrow \ldots$$ is still acyclic. Thus, in particular, every Ding injective module is Gorenstein injective, and, if $R$ is a noetherian ring, then the two classes coincide.

Another generalization of the Gorenstein injectives are the Gorenstein AC-injective modules. They were introduced in \cite{gillespie:14:stable}. We recall their definition.\\
First, a module $F$ is said to be \emph{of type $FP_{\infty}$} if it has a projective resolution
$$\ldots \rightarrow P_2 \rightarrow P_1 \rightarrow P_0 \rightarrow F \rightarrow 0$$ with each $P_i$ finitely generated. \\
We also recall that an $R$-module $A$ is \emph{absolutely clean} if $Ext^1(F,A) = 0$ for all $R$-modules $F$ of type $FP_{\infty}$.\\

An $R$-module $M$ is called \emph{Gorenstein AC-injective} if there exists an exact complex of injective modules
$$I = \ldots \rightarrow I_1 \rightarrow I_0 \rightarrow I_{-1} \rightarrow \ldots$$
with $M = Ker(I_0 \rightarrow I_{-1})$ and such that the complex $Hom(A,I)$ is still exact for any absolutely clean module $A$.

The projectively coresolved Gorenstein flat modules ($\mathcal{PGF}$ modules, for short) were recently introduced (in \cite{saroch:18:gor}) by Saroch and Stovicek. A module $M$ is in the class $\mathcal{PGF}$ if it is a cycle of an acyclic complex $$\ldots \rightarrow P_1 \rightarrow P_0 \rightarrow P_{-1} \rightarrow \ldots$$ consisting of projective modules which remains exact after tensoring by arbitrary injective $R$-modules.

Throughout the paper we will use the notation $Proj$ ($Inj$ respectively) for the class of projective (injective respectively) modules. The class of Gorenstein projective (Gorenstein injective, Gorenstein flat respectively) modules is denoted $\mathcal{GP}$ ($\mathcal{GI}$, $\mathcal{GF}$ respectively). The class of Ding injective modules is denoted $\mathcal{DI}$, while the class of Gorenstein AC-injective modules is denoted $\mathcal{GI}_{ac}$.

A pair of classes $(\mathcal{A}, \mathcal{B})$ in an abelian category $\mathcal D$ is called a {\it cotorsion pair} if $\mathcal{B}=\mathcal{A}^{\perp}$ and $\mathcal{A}=^{\perp}\!\!\! \mathcal{B}$, where for a given class of objects $\mathcal C$, the right orthogonal class $\mathcal C^{\perp}$ is defined to be the class of objects $Y$ in $\mathcal D$ such that $\mathrm{Ext}^1(A,Y)=0$ for all $A\in \mathcal{C}$. Similarly, we define the left orthogonal class $^{\perp}{\mathcal C}$ to be the class of objects $L$ such that $\mathrm{Ext}^1(L,C)=0$ for all $C \in \mathcal{C}$. The cotorsion pair is called {\it hereditary} if $\mathrm{Ext}^i(A,B)=0$ for all $A\in\mathcal A$, $B\in \mathcal B$, and $i\geq 1$. The cotorsion pair is called {\it complete} if for each object $M$ in $\mathcal D$ there exist short exact sequences $0\to M\to B\to A\to 0$ and $0\to B'\to A'\to M\to 0$ with $A,A'\in \mathcal A$ and $B,B'\in \mathcal B$.

Following Gillespie \cite[Definition 3.3 and Proposition 3.6]{Gill04} there are four classes of complexes in $Ch(R)$ that are associated with a cotorsion pair $(\mathcal{A}, \mathcal{B})$ in $R$-Mod:
\begin{enumerate}
\item An acyclic complex $X$ is an \emph{$\mathcal{A}$-complex} if $Z_j(X) \in \mathcal{A}$ for all integers $j$. We denote by $\widetilde{\mathcal{A}}$ the class of all acyclic $\mathcal A$-complexes.
\item An acyclic complex $U$ is a \emph{$\mathcal{B}$-complex} if $Z_j(X) \in \mathcal{B}$ for all integers $j$. We denote by $\widetilde{\mathcal{B}}$ the class of all acyclic $\mathcal B$-complexes.
\item A complex $Y$ is a \emph{dg-$\mathcal{A}$ complex} if each $Y_n \in \mathcal{A}$ and each map $Y\to U$ is null-homotopic, for each complex $U\in \widetilde{\mathcal{B}}$. We denote by $ dg (\mathcal{A})$ the class of all dg-$\mathcal{A}$ complexes.
\item A complex $W$ is a \emph{dg-$\mathcal{B}$ complex} if each $W_n \in \mathcal{B}$ and each map $V\to W$ is null-homotopic, for each complex $V\in \widetilde{\mathcal{A}}$. We denote by $ dg (\mathcal{B})$ the class of all dg-$\mathcal{B}$ complexes.
\end{enumerate}

Yang and Liu showed in \cite[Theorem 3.5]{yang:11:cotorsion} that when $(\mathcal{A}, \mathcal{B})$ is a complete hereditary cotorsion pair in $R$-Mod, the pairs $(dg (\mathcal{A}), \widetilde{\mathcal{B}})$ and $(\widetilde{\mathcal{A}}, dg (\mathcal{B}))$ are complete (and hereditary) cotorsion pairs. Moreover, by Gillespie \cite[Theorem 3.12]{Gill04}, we have that  $\widetilde{\mathcal{A}}=dg (\mathcal{A})\cap \mathcal E$ and $ \widetilde{\mathcal{B}}=dg (\mathcal{B})\cap \mathcal E$ (where $\mathcal E$ is the class of all acyclic complexes). For example, from the (complete and hereditary) cotorsion pairs $(Proj,R\textrm{-Mod})$ and $(R\textrm{-Mod},Inj)$ one obtains the standard (complete and hereditary) cotorsion pairs $(\mathcal E,dg(Inj))$ and $(dg(Proj),\mathcal E)$. Over any ring $R$, the pair $(^\bot \mathcal{GI}, \mathcal{GI})$ is a complete hereditary cotorsion pair. This is due to Saroch and Stovicek (\cite{saroch:18:gor}. Therefore $(dg (^\bot \mathcal {GI}), \widetilde{\mathcal {GI}})$ is a complete cotorsion pair in $Ch(R)$. \\

Also by Saroch and Stovicek (\cite{saroch:18:gor}), over any ring $R$, $(\mathcal{GF},\mathcal{GF}^{\perp})$ is a complete hereditary cotorsion pair. The class $\mathcal{GF}^{\perp}$ is known as the class of {\it Gorenstein cotorsion modules} and it is usually denoted by $\mathcal{GC}$. So the pair $(\widetilde{\mathcal{GF}}, dg (\mathcal{GC}))$ is a complete cotorsion pair.\\ And, also by \cite{saroch:18:gor}, over any ring $R$, the class of projectively coresolved Gorenstein flat modules is the left half of a complete hereditary cotorsion pair. Hence the pair $(\widetilde{\mathcal{PGF}}, dg (\mathcal{PGF}^\bot))$ is a complete cotorsion pair.

We recall that, given a class of modules $\mathcal{A}$, $dw(\mathcal{A})$ denotes the class of complexes of modules, $X$, such that each component, $X_n$, is in $\mathcal{A}$.\\

\section{main results}
We start with the following lemma.
\begin{lemma}\label{lemma1}
Let $(\mathcal{A}, \mathcal{B})$ be a complete hereditary cotorsion pair in $R$-Mod. The following assertions are equivalent:
\begin{enumerate}
\item Every exact complex with each component in $\mathcal A\cap \mathcal B$ has cycles in $\mathcal B$.
\item Every exact complex in $dw(\mathcal{B})$ is in $\widetilde{\mathcal B}$.
\item Every complex in $dw(\mathcal{B})$ is in $dg(\mathcal B)$.
\end{enumerate}
\end{lemma}
\begin{proof}
(1) $\Rightarrow$ (2) Let $B\in dw(\mathcal B)\cap \mathcal E$. We have to prove that each $Z_n B$ is in $\mathcal B$. Equivalently $\mathrm{Ext}^1_R(A,Z_nB)=0$, for each $A\in \mathcal A$. Since $(\widetilde{\mathcal{A}},dg(\mathcal B))$ is a complete cotorsion pair in $Ch(R)$, there exists a short exact sequence $$0\to B\to T\to L\to 0, $$ with $T\in dg(\mathcal{B})$ and $L\in \widetilde{\mathcal{A}}$. Since $B$ and $L$ are exact complexes, it follows that $T$ is also exact. Hence $T\in \widetilde{\mathcal{B}}$. Moreover, since each $B_n$ and $T_n$ are in $\mathcal{B}$ and the cotorsion pair $(\mathcal{A},\mathcal B)$ is hereditary, we get that $L_n\in \mathcal B$. Therefore $L\in \widetilde{\mathcal{A}}\cap dw(\mathcal{B})$. But now, condition $(1)$ tells us that $$\widetilde{\mathcal{A}}\cap dw(\mathcal{B})=\widetilde{\mathcal{A}}\cap \mathcal{E}\cap dw(\mathcal{B})=\widetilde{\mathcal{A}}\cap\widetilde{\mathcal{B}}, $$ i.e. $L$ is a contractible complex with components in $\mathcal{A}\cap\mathcal{B}$.\\ Now, since $B$ is an exact complex, for each $n\in \mathbb{Z}$, we have a short exact sequence $$0\to Z_nB\to Z_nT\to Z_nL\to 0. $$ Let us fix $A\in \mathcal A$. Since $T\in \widetilde{\mathcal B}$ (so $Z_nT$ is in particular in $\mathcal B$) we have the exact sequence of abelian groups $$\mathrm{Hom}(A,Z_nT)\to \mathrm{Hom}(A,Z_nL)\to \mathrm{Ext}^1(A,Z_nB)\to \mathrm{Ext}^1(A,Z_nT)=0. $$ So to prove $\mathrm{Ext}^1(A,Z_nB)=0$, it suffices to show that every map $A\to Z_nL$ can be lifted to a map $A\to Z_nT$. But this follows easily, since the  short exact sequences
$$0\to Z_{n+1}L\to L_{n+1}\to Z_n L\to 0 $$ and $$0\to B_{n+1}\to T_{n+1}\to L_{n+1}\to 0$$ are both split exact.\\
(2) $\Rightarrow$ (3) Let $B\in dw(\mathcal B)$. Let us fix a complex $V\in \widetilde{\mathcal A}$. We have to prove that each map $V\to B$ is null homotopic. Since $(\widetilde{\mathcal{A}},dg(\mathcal B))$ is a complete cotorsion pair in $Ch(R)$, there exists a short exact sequence $$0\to B\to T\to L\to 0, $$ with $T\in dg(\mathcal{B})$ and $L\in \widetilde{\mathcal{A}}$. In particular, every map $V\to T$ is null-homotopic. Since $(\mathcal{A},\mathcal B)$ is a hereditary cotorsion pair and $B_n,T_n\in \mathcal{B}$, we get that $L\in \widetilde{\mathcal{A}}\cap dw(\mathcal{B})$. Now condition (2) tells us that $L\in \widetilde{\mathcal A}\cap \widetilde{\mathcal B}$, i.e. $L$ is a contractible complex with components in $\mathcal{A}\cap\mathcal{B}$. But then $B\to T$ is a homotopy isomorphism, and so every map $V\to B$ is also null-homotopic. Hence $B\in dg(\mathcal B)$.\\
(3) $\Rightarrow$ (1) This is immediate, because the equality $dg(\mathcal{B})\cap \mathcal{E}=\widetilde{\mathcal{B}}$ always holds.
\end{proof}
We have the corresponding dual version of the previous lemma:
\begin{lemma}\label{lemma1dual}
Let $(\mathcal{A}, \mathcal{B})$ be a complete hereditary cotorsion pair in $R$-Mod. The following assertions are equivalent:
\begin{enumerate}
\item Every exact complex with each component in $\mathcal A\cap \mathcal B$ has cycles in $\mathcal A$.
\item Every exact complex in $dw(\mathcal{A})$ is in $\widetilde{\mathcal A}$.
\item Every complex in $dw(\mathcal{A})$ is in $dg(\mathcal A)$.
\end{enumerate}
\end{lemma}

Theorem 4 below gives equivalent characterizations of the condition that every exact complex of injective $R$-modules is totally acyclic.\\

\begin{theorem}
Let $R$ be any ring. The following are equivalent:
\begin{enumerate}
\item Every exact complex of injective $R$-modules is totally acyclic.
\item Every exact complex of Gorenstein injective $R$-modules is in $\widetilde{\mathcal{GI}}$.
\item Every complex of Gorenstein injective $R$-modules is a dg-Gorenstein injective complex.
\end{enumerate}
\end{theorem}

\begin{proof}
This follows by applying Lemma \ref{lemma1} to the complete hereditary cotorsion pair $(^\bot \mathcal{GI}, \mathcal{GI})$, taking into account that
$^\bot \mathcal{GI}\cap \mathcal{GI}$ is just the class of all injective modules.
\end{proof}

Bravo, Gillespie and Hovey proved (\cite{gillespie:14:stable}, proof of Theorem 5.5) that over any ring $R$, the class of Gorenstein AC-injective modules form the left half of a complete hereditary cotorsion pair, $(^\bot \mathcal{GI}_{ac}, \mathcal{GI}_{ac})$. Then a similar argument as above gives the following:\\

\begin{theorem}
Let $R$ be any ring. The following are equivalent:
\begin{enumerate}
\item Every exact complex of injective $R$-modules has all its cycles Gorenstein AC-injective modules.
\item Every exact complex of Gorenstein AC-injective $R$-modules is in $\widetilde{\mathcal{GI}_{ac}}$.
\item Every complex of Gorenstein AC-injective $R$-modules is a dg-Gorenstein AC-injective complex.
\end{enumerate}
\end{theorem}
In particular, when $R$ is a coherent ring, the class of Gorenstein AC-injectives coincides with that of Ding injective modules, $\mathcal{DI}$. So we obtain the following\\
\begin{corollary}
Let $R$ be a coherent ring. The following are equivalent:
\begin{enumerate}
\item Every exact complex of injective $R$-modules has all its cycles Ding injective modules.
\item Every exact complex of Ding injective $R$-modules is in $\widetilde{\mathcal{DI}}$.
\item Every complex of Ding injective $R$-modules is a dg-Ding injective complex.
\end{enumerate}
\end{corollary}

The following is a consequence of Lemma \ref{lemma1dual}.
\begin{theorem}
Let $R$ be any ring. Then the following are equivalent.
\begin{enumerate}
\item Every exact complex of flat right $R$-modules is F-totally acyclic.
\item Every exact complex of cotorsion-flat right $R$-modules is F-totally acyclic
\item Every exact complex of Gorenstein flat right $R$-modules is in $\widetilde{\mathcal{GF}}$.
\item  Every complex of Gorenstein flat right $R$-modules is a dg-Gorenstein flat complex.
\end{enumerate}
\end{theorem}

\begin{proof}
To get that conditions (2), (3), and (4) are equivalent, we just need to apply Lemma \ref{lemma1dual} to the complete hereditary cotorsion pair
$(\mathcal{GF},\mathcal{GC})$, by taking into account that $\mathcal{GF}\cap\mathcal{GC} $ is the class of all cotorsion-flat right $R$-modules.\\ Finally, (1) $\Rightarrow$ (2) is immediate, and (4) $\Rightarrow$ (1) follows because $dg(\mathcal{GF})\cap \mathcal E=\widetilde{\mathcal{GF}}$.
\end{proof}




By Saroch-Stovicek's recent results, over any ring $R$ there is a complete hereditary cotorsion pair $(\mathcal{PGF}, \mathcal{PGF}^\bot)$, with $\mathcal{PGF}\cap \mathcal{PGF}^\bot$ the class of all projective modules. Then Lemma \ref{lemma1dual} also gives the following result:\\
\begin{theorem}
Let $R$ be any ring. The following are equivalent:\\
(1) Every exact complex of projective modules has all its cycles in $\mathcal{PGF}$.\\
(2) Every exact complex of projectively coresolved Gorenstein flat modules has all its cycles in $\mathcal{PGF}$.\\
(3) Every complex of projectively coresolved Gorenstein flat modules is a dg-$\mathcal{PGF}$ complex.
\end{theorem}

We show that if the ring $R$ has the property that there is a nonnegative integer $n$ such that every Gorenstein flat module has projectively coresolved Gorenstein flat dimension less than or equal to $n$, then the statements from Theorem 7 are equivalent with those from Theorem 8. We use the notation $Pgf.d. A$ to denote the $\mathcal{PGF}$ dimension of the module $A$, and $Gfd. A$ to denote the Gorenstein flat dimension of the module $A$. We will use the following remark.

\begin{remark}
 If $A$ is a module of finite $\mathcal{PGF}$ dimension then $A$ has finite Gorenstein flat dimension and $Gfd. A \le Pgf.d. A$ (because $\mathcal{PGF} \subseteq \mathcal{GF}$).
\end{remark}

\begin{theorem}
Let $R$ be a ring. Assume that there is a nonnegative integer $n$ such that every Gorenstein flat module has $\mathcal{PGF}$ dimension less than or equal to $n$. Then the following are equivalent:\\
(1) Every exact complex of flat $R$-modules is F-totally acyclic. \\ 
(2) Every exact complex of Gorenstein flat $R$-modules is in $\widetilde{\mathcal{GF}}$.\\
(3) Every complex of Gorenstein flat $R$-modules is a dg-Gorenstein flat complex.\\
(4) Every exact complex of cotorsion-flat right $R$-modules is F-totally acyclic \\
(5) Every exact complex of projective modules has all its cycles in $\mathcal{PGF}$.\\
(6) Every exact complex of projectively coresolved Gorenstein flat modules has all its cycles in $\mathcal{PGF}$.\\
(7) Every complex of projectively coresolved Gorenstein flat modules is a dg-$\mathcal{PGF}$ complex.
\end{theorem}

\begin{proof}
Statements (1), (2), (3) and (4) are equivalent by Theorem 7. And (5), (6) and (7) are equivalent by Theorem 8.\\
(1) $\Rightarrow$ (5) By hypothesis there is $n \ge 0$ such that any Gorenstein flat module has $\mathcal{PGF}$ dimension at most $n$. Let $X$ be an exact complex of projective modules. By (1), $Z_l(X) \in \mathcal{GF}$, so $Pgf.d. Z_l (X) \le n$. Thus $Z_{l+n}(X) \in \mathcal{PGF}$ for all $l$. By replacing $l$ with $l-n$ we obtain that $Z_l (X) \in \mathcal{PGF}$ for all $l$.\\
(6) $\Rightarrow$ (1) Let $F$ be an exact complex of flat modules and consider a partial projective resolution of $F$: $$0 \rightarrow C \rightarrow P_{n-1} \rightarrow \ldots P_0 \rightarrow F \rightarrow 0$$ Then $C$ is an exact complex. Since each $F_j$ is Gorenstein flat, so $Pgf.d. F_j \le n$ and since every $P_{i,j}$ is a projective module, it follows that each $C_j \in \mathcal{PGF}$. Then by (6) we have that $Z_l (C) \in \mathcal{PGF}$ for all $l$. Thus $Pgf.d. (Z_l(F)) \le n$ for all $l$, and therefore, by Remark 9 above, $Gfd. Z_lF \le n$ for all $l$. In particular we obtain that $Z_{l+n} (F) \in \mathcal{GF}$ for all $l$, and so, that all cycles of $F$ are Gorenstein flat modules.
\end{proof}

We recall that a ring $R$ is left $n$-perfect($n \ge 0$) if every flat left $R$-module has projective dimension less than or equal to $n$ (this is \cite{EJL}, Definition 2.1). We show that every left $n$-perfect ring satisfies the hypothesis of Theorem 10.\\ In the following $pd M$ denotes the projective dimension of the module $M$.

\begin{lemma}
Let $R$ be a left $n$-perfect ring. Then every Gorenstein flat $R$-module has $\mathcal{PGF}$ dimension less than or equal to $n$.

\end{lemma}

\begin{proof}
Let $G \in \mathcal{GF}$. Since $G \in \mathcal{GF}$, there exists an F-totally acyclic complex $X$ such that $G = Z_0(X)$. Consider a partial projective resolution of $X$:
$$0 \rightarrow C \rightarrow P_{n-1} \rightarrow \ldots \rightarrow P_0 \rightarrow F \rightarrow 0.$$ Then $C$ is an exact complex. Since $pd. F_j \le n$, it follows that each $C_j$ is a projective module.

The exact sequence $$0 \rightarrow Z_l(C) \rightarrow Z_l(P_{n-1}) \rightarrow \ldots \rightarrow Z_l(P_0) \rightarrow Z_l(X) \rightarrow 0$$ with $Z_l(X) \in \mathcal{GF}$ and with $Z_l(P_j) \in Proj$ (for all $j$), gives that $Z_l(C) \in \mathcal{GF}$ for all $l$ (since by \cite{saroch:18:gor}, Corollary 3.12, the class of Gorenstein flat modules is closed under kernels of epimorphisms). Thus $C$ remains exact when tensoring it with any injective module. Since $C$ is an exact complex of projectives that is also $Inj \otimes - $ exact, we have that $Z_l(C) \in \mathcal{PGF}$ for all $l$. The same exact sequence, $$0 \rightarrow Z_l(C) \rightarrow Z_l(P_{n-1}) \rightarrow \ldots \rightarrow Z_l(P_0) \rightarrow Z_l(X) \rightarrow 0$$ gives that $Pgf. d. Z_l(X) \le n$, for every $l$. In particular, $Pgf. d. G \le n$.
\end{proof}

Thus Theorem 10 gives the following result:\\

\begin{theorem}
Let $R$ be a left $n$-perfect ring. The following statements are equivalent:
\begin{enumerate}
\item Every exact complex of flat (right) $R$-modules is F-totally acyclic.
\item Every exact complex of Gorenstein flat (right) $R$-modules is in $\widetilde{\mathcal{GF}}$.
\item Every complex of Gorenstein flat (right) $R$-modules is a dg-Gorenstein flat complex.
\item Every exact complex of projective (right) $R$-modules has all its cycles in $\mathcal{PGF}$.
\item Every exact complex of $\mathcal{PGF}$  (right) $R$-modules is in $\widetilde{\mathcal{PGF}}$.
\item Every complex of $\mathcal{PGF}$ modules is a dg-$\mathcal{PGF}$ complex.
\end{enumerate}
\end{theorem}

We recall that any commutative noetherian ring of finite Krull dimension is an $n$-perfect ring (by \cite{EJL}, Example 2.3)

\begin{corollary}
Let $R$ be a commutative noetherian ring of finite Krull dimension. The following statements are equivalent:
\begin{enumerate}
\item Every exact complex of injective $R$-modules is totally acyclic.
\item Every exact complex of flat $R$-modules is F-totally acyclic.
\item Every exact complex of Gorenstein flat $R$-modules is in $\widetilde{\mathcal{GF}}$.
\item Every exact complex of Gorenstein injective $R$-modules is in $\widetilde{\mathcal{GI}}$.
\item Every complex of Gorenstein injective $R$-modules is a dg-Gorenstein injective complex.
\item Every complex of Gorenstein flat $R$-modules is a dg-Gorenstein flat complex.
\item Every exact complex of projective $R$-modules has all its cycles in $\mathcal{PGF}$.
\item Every exact complex of $\mathcal{PGF}$ $R$-modules is in $\widetilde{\mathcal{PGF}}$.
\item Every complex of $\mathcal{PGF}$ modules is a dg-$\mathcal{PGF}$ complex.
\item $R$ is a Gorenstein ring.
\end{enumerate}
\end{corollary}

\begin{proof}
Statements (1), (2), (3), (4), (5), (6) and (10) are equivalent by \cite{iacob:17:totally}, Corollary 1, and (2), (7), (8) and (9) are equivalent by Theorem 12 above.
\end{proof}
We show next that over a coherent ring $R$ with the property that $\mathcal{DI}^+ \subseteq \mathcal{GF}$, every module $M$ such that $M^+$ is Gorenstein flat is a cycle of an exact complex of FP-injective modules. We will use the following:\\

\begin{lemma}
Let $R$ be any ring. If $A \in ^\perp \mathcal{DI}$ then $A^+ \in \mathcal{GC}$.
\end{lemma}

\begin{proof}
Let $A \in ^\perp \mathcal{DI}$, and let $B$ be any Gorenstein flat module. We have $\mathrm{Ext}^1(B, A^+) \simeq \mathrm{Ext}^1(A, B^+) = 0$ since $B^+$ is Ding injective (by \cite{chen:14:coherent}, Proposition 5.3.4).
\end{proof}


\begin{lemma}
Let $R$  be a coherent ring such that $\mathcal{DI}^+ \subseteq \mathcal{GF}$. If $M^+$ is Gorenstein flat then $M$ has an exact left FP-injective resolution, and so, $M$ is a cycle of an acyclic complex of FP-injective modules.
\end{lemma}



\begin{proof}
Let $M$ be such that its character module, $M^+$, is Gorenstein flat.

Since $R$ is coherent, the Ding injectives coincide with the Gorenstein AC-injectives, so $(^\perp \mathcal{DI}, \mathcal{DI})$ is a complete hereditary cotorsion pair. Thus, there is an exact sequence $$0 \rightarrow A \rightarrow C \rightarrow M \rightarrow 0$$ with $C \in ^\perp \mathcal{DI}$ and with $A$ Ding injective. This gives an exact sequence $$0 \rightarrow M^+ \rightarrow C^+ \rightarrow A^+ \rightarrow 0$$ with both $M^+$ and $A^+$ being Gorenstein flat. By \cite{saroch:18:gor}, Corollary 3.12, the class $\mathcal{GF}$ is closed under extensions. So $C^+$ is Gorenstein flat . By Lemma 14, $C^+$ is also Gorenstein cotorsion, so it is in $\mathcal{GF} \bigcap \mathcal{GC}$. Thus $C^+$ is flat. Since $R$ is coherent, this implies that $C$ is FP-injective (by \cite{chen:14:coherent}, Theorem 2.2.13). Thus $M$ has a special surjective FP-injective precover with kernel $A \in \mathcal{DI}$. Since $A$ has an exact FP-injective left resolution, by pasting them together, we obtain an exact left FP-injective resolution of $M$. Every $R$-module, $M$ in particular, has an exact right FP-njective resolution. By pasting them together we obtain that $M$ is a cycle of an acyclic complex of FP-injective modules.
\end{proof}

We can prove now:\\

\begin{theorem}
Let $R$ be a coherent ring and let $n < \infty$ be a non negative integer.

The following statements are equivalent:\\
(1) Every exact complex of FP-injective modules has all its cycles of Ding injective dimension at most $n$.\\
(2) Every exact complex of injective modules has all its cycles Ding injective modules, and and every $R$-module $M$ whose character module, $M^+$, is Gorenstein flat, has Ding injective dimension at most $n$.\\
(3) Every exact complex of flat right $R$-modules is F-totally acyclic, and every $R$-module $M$ whose character module, $M^+$, is Gorenstein flat, has Ding injective dimension at most $n$.\\
\end{theorem}

\begin{proof}
1() $\Rightarrow$ (2) Let $X = \ldots \rightarrow X_1 \rightarrow X_0 \rightarrow X_{-1} \rightarrow \ldots$ be an exact complex of injective modules. By (1), each cycle, $Z_l(X)$ has Ding injective dimension at most $n$. Since each $X_j$ is injective it follows that $Z_{l-n}(X)$ is Ding injective, for all $l$. By replacing $l$ with $l+n$ we obtain that $Z_l(X)$ is Ding injective for all $l$.\\
Let $Y$ be an exact complex of flat modules. Then $Y^+$ is an exact complex of injectives (by \cite{enochs:00:relative}, Theorem 3.2.9), so $Z_l(Y)^+$ is Ding injective for all $l$. Since the ring is coherent, it follows that each cycle of $Y$ is Gorenstein flat (by \cite{holm:05:gor.dim}, Theorem 3.6). Thus, if (1) holds, then every exact complex of flat modules is F-totally acyclic.

Let $G$ be a Ding injective module. Then $G$ is a cycle of an acyclic complex of injective modules, $I$. Since $I^+$ is an exact complex of flat right $R$-modules (by \cite{chen:14:coherent}, Theorem 2.2.13), it follows, by the above, that $I^+$ is F-totally acyclic. Thus $Z_l (I^+)$ is Gorenstein flat for all $l$ and so, in particular $G^+$ is Gorenstein flat. So if (1) holds then the ring $R$ satisfies the hypotheses of Lemma 15.\\
Let $M$ be such that $M^+$ is Gorenstein flat. By Lemma 15, $M$ is a cycle of an exact complex of FP-injective modules. By (1), $M$ has Ding injective dimension at most $n$.\\
(2) $\Rightarrow$ (3) If $Y$ is an exact complex of flat right $R$-modules, then $Y^+$ is an exact complex of injective modules, so $Z_l(Y)^+$ is Ding injective for all $l$. Since the ring is coherent, it follows that each cycle of $Y$ is Gorenstein flat.\\
The second statement of (3) follows immediately from (2).\\
 (3) $\Rightarrow$ (1) Let $X$ be an exact complex of FP-injective modules. Since the ring is coherent, we have that $X^+$ is an exact complex of flat modules (by \cite{chen:14:coherent}, Theorem 2.2.13). By (3), it is F-totally acyclic, so $Z_l(X)^+$ is Gorenstein flat, for all l. By (3), each $Z_l(X)$ has Ding injective dimension at most $n$.

\end{proof}

A similar argument to the proof of Theorem 16 gives the following characterization of Gorenstein rings.


\begin{theorem}
Let $R$ be a commutative coherent ring. 
The following statements are equivalent:\\
(1) Every exact complex of FP-injective modules has all its cycles Ding injective modules.\\
(2) Every exact complex of injective modules has all its cycles Ding injective modules, and and every $R$-module $M$ whose character module, $M^+$, is Gorenstein flat, is Ding injective.\\
(3) Every exact complex of flat modules is F-totally acyclic, and every $R$-module $M$ whose character module, $M^+$, is Gorenstein flat, is a Ding injective module.\\
If moreover $R$ has finite Krull dimension, then (1), (2) and (3) are also equivalent to\\
(4) $R$ is a Gorenstein ring.
\end{theorem}

\begin{proof} Statements (1), (2) and (3) are equivalent by Theorem 16, for the case $n=0$.\\
Now assume that $R$ has finite Krull dimension.\\
(1) $\Rightarrow$ (4). Let $(E_i)_i$ be a family of injective $R$-modules, and let $E = \oplus E_i$. Since each $E_i$ is FP-injective and $R$ is coherent, $E$ is FP-injective. Then $0 \rightarrow E = E \rightarrow 0$ is an exact complex of FP-injective modules, so by (1), $E = Ker(E \rightarrow 0)$ is a Ding injective module. Also, each $E_i$ is in $^\bot \mathcal{DI}$, and so $E = \oplus E_i \in ^\bot \mathcal{DI}$. Thus $E \in \mathcal{DI} \bigcap ^\bot \mathcal{DI}$, so $E$ is injective (\cite{enochs:00:relative}). Since the class of injective modules is closed under direct sums, the ring $R$ is noetherian (\cite{enochs:00:relative}, Theorem 3.1.17). Since $R$ is commutative noetherian of finite Krull dimension, and since (1) implies that every exact complex of injectives is totally acyclic, it follows that $R$ is a Gorenstein ring (by \cite{iacob:17:totally}, Corollary 2).\\
(4) $\Rightarrow$ (3) By \cite{enochs:00:relative}, Theorem 12.3.1, if $R$ is a Gorenstein ring then every every exact complex of flat $R$-modules is F-acyclic. Also, a Gorenstein ring satisfies the hypothesis of Lemma 13. So, if $M^+$ is Gorenstein flat, then $M$ is a cycle of an exact complex of FP-injective modules. Over a noetherian ring these are the same as the injective modules. Since $R$ is Gorenstein, every cycle of an exact complex of injectives is Ding injective.

\end{proof}







\bibliographystyle{plain}



\bibliographystyle{plain}

\end{document}